\documentclass[12pt]{article}
\usepackage{amsmath,amsthm,amsfonts}
\usepackage{eucal}

\newtheorem{theorem}{Theorem}
\newtheorem{lemma}[theorem]{Lemma}
\newtheorem{corollary}[theorem]{Corollary}
\newtheorem{proposition}[theorem]{Proposition}
\theoremstyle{remark}

\newtheorem{example}{Example}[section]

\renewcommand{\Pr}{{\mathbb P}}
\newcommand{\E}{{\mathbb E}}
\newcommand{\R}{{\mathbb R}}

\newcommand{\Z}{{\mathbb Z}}

\begin{document}
\begin{center}
  \Large\textbf{Tail asymptotics for the supercritical\\
  Galton--Watson process in the heavy-tailed case
  \footnote{Supported by DFG}}
\end{center}

 \begin{center}
   Denis Denisov, Dmitry Korshunov, and Vitali Wachtel
 \end{center}

 \begin{center}
   \textit{University of Manchester, Sobolev Institute of Mathematics and University of Munich}
 \end{center}

\begin{abstract}
As well known, for a supercritical Galton--Watson process $Z_n$
whose offspring distribution has mean $m>1$,
the ratio $W_n:=Z_n/m^n$
has a.s. limit, say $W$. We study tail behaviour of
the distributions of $W_n$ and $W$ in the case where
$Z_1$ has heavy-tailed distribution, that is,
$\E e^{\lambda Z_1}=\infty$ for every $\lambda>0$.
We show how different types of distributions of $Z_1$ lead
to different asymptotic behaviour of the tail of $W_n$ and $W$.
We  describe the most likely way how large values of the process
occur.


\end{abstract}

\section{Introduction}
\label{sec:introduction}

Let $Z_n$ be a supercritical Galton--Watson process
with $Z_0=1$, $m:=\E Z_1>1$. By definition,
$$
Z_{n+1}=\sum_{i=1}^{Z_n}\xi_i^{(n)},
$$
where $\xi_i^{(n)}$, $i$, $n=0$, $1$, \ldots, are
independent identically distributed random variables
with distribution $F$ on $\Z^+=\{0,1,2,\ldots\}$;
by $\overline F(x)$ we denote the tail of $F$,
$\overline F(x):=\Pr\{\xi>x\}$.

Put $W_n:=Z_n/m^n$. As well known (see, e.g., \cite[Theorem 1.6.1]{AN})
$W_n\to W$ a.s. as $n\to\infty$. If $\E\xi\log\xi<\infty$
then $\E W=1$, so $\Pr\{W>0\}>0$, see \cite[Theorem 1.10.1]{AN}.

Our goal is to consider asymptotic probabilities for the martingale
sequence $\{W_n\}$ and for its limit $W$. More precisely, we are going
to find asymptotics for $\Pr\{W_n>x\}$
as $x\to\infty$ in the whole range of $n\ge 1$.

The tail-behaviour of the martingale limit is one of the classical
problems in the theory of supercritical Galton-Watson processes. The
study of $\Pr\{W>x\}$ has been initiated by Harris \cite{H48} who showed
that if $\xi$ is bounded, then
$$
\log \E e^{uW}=u^\gamma H(u)+O(1)\quad\mbox{as }u\to\infty,
$$
where $H$ is a positive multiplicatively periodic function and $\gamma$ is
defined by the equality $m^\gamma=\max\{k:\,\Pr\{\xi=k\}>0\}$. This
information on the generating function can be translated into asymptotics
of tail-probabilities. It was done by Biggins and Bingham \cite{BB}:
\begin{equation}\label{HBB}
\log \Pr\{W>x\}\sim -x^{\gamma/(\gamma-1)} M(x),
\end{equation}
where $M$ is also a positive multiplicatively periodic function;
hereinafter we write $f(x)\sim g(x)$ as $x\to\infty$ if
$f(x)/g(x)\to 1$.
Bingham and Doney \cite{BD1,BD2} found asymptotics for $\Pr\{W>x\}$
in the case when $\xi$ is regularly varying with
non-integer index $\alpha<-1$ (for the case of integer $\alpha$
see De Meyer \cite{M}). In \cite{BB} one can find
similar to \eqref{HBB} results for the left tail of $W$ in the case, when
the minimum offspring size is at least $2$. Fleischmann and Wachtel
\cite{FW07,FW09} found exact (without logarithmic scaling) asymptotics for
$\Pr\{W_n\in(0,x)\}$ and $\Pr\{W\in(0,x)\}$ as $x\to0$. These two papers
give a complete description of the asymptotic behaviour of the left tail of $W$.
It is possible to adapt the method from \cite{FW09} to upper deviation
problems for processes with polynomial offspring generating functions. As a
result one gets exact asymptotics for $\Pr\{W>x\}$ as $x\to\infty$, see
Remark 3 in \cite{FW09}.

In all the papers mentioned above, the proofs were based on the fact that
$\varphi(u):=\E e^{-uW}$ satisfies the Poincare functional equation,
$\varphi(mu)=f(\varphi(u))$, where $f$ stands for the offspring generating
function. In the present paper we do not use that equation. Instead, we
apply probabilistic techniques for sums of independent
identically distributed variables and for Galton--Watson
processes with heavy tails which were developed in recent years.

We work with the following classes of distributions.

Distribution of a random variable $\xi$ is called {\it heavy-tailed}
if $\E e^{\lambda\xi}=\infty$ for every $\lambda>0$.

We say that a distribution $F$ on $\R$ is {\it dominated varying},
and write $F\in{\mathcal D}$, if
\begin{eqnarray}\label{class.D}
\sup_{x}\frac{\overline F(x/2)}{\overline F(x)}<\infty.
\end{eqnarray}

A distribution $F$ on $\R$ is called {\it intermediate regularly
varying} if
$$
\lim_{\varepsilon\downarrow0}\liminf_{x\to\infty}
\frac{\overline F(x(1+\varepsilon))}{\overline F(x)} = 1.
$$
Note that any regularly varying distribution is intermediate
regularly varying. Any intermediate regularly varying distribution
is dominated varying.

For any positive function $h(x)\to\infty$, we say that $F$ is
$h$-{\it insensitive} if
$\overline F(x+h(x))\sim \overline F(x)$ as $x\to\infty$.
A distribution $F$ is intermediate regularly varying
if and only if $F$ is $h$-insensitive for any positive
function $h$ such that $h(x)=o(x)$ as $x\to\infty$;
in other words, if $F$ is $o(x)$-insensitive
(see \cite[Theorem 2.47]{FKZ}).

We say that a distribution $F$ on $\R^+$ with mean $m$ is
{\it strong subexponential}, and write $F\in{\mathcal S}^*$, if
$$
\int_0^x \overline F(x-y)\overline F(y)dy \sim 2m\overline F(x)
\quad\mbox{ as }x\to\infty.
$$
Among strong subexponential distributions are intermediate
regularly varying, log-normal and Weibull with parameter $\beta<1$.
Any dominated varying distribution is in ${\mathcal S}^*$
if it is long-tailed, that is, constant-insensitive.

A distribution $F$ is called {\it rapidly varying} if,
for any $\varepsilon>0$,
$$
\overline F(x(1+\varepsilon))=o(\overline F(x))
\quad\mbox{ as }x\to\infty.
$$
Clearly this class includes Weibull distributions
$\overline F(x)=e^{-x^\beta}$ with parameter $\beta>0$.
The log-normal distribution is also rapidly varying.
This class does not include intermediate regularly varying
distributions.

\begin{theorem}\label{thm:irv}
Let $F$ be dominated varying distribution
such that, for some $\delta>0$ and $c<\infty$,
\begin{eqnarray}\label{cond:mat}
\overline F(xy) &\le& c\overline F(x)/y^{1+\delta}
\quad\mbox{ for all }x,\ y>1.
\end{eqnarray}
Then there exist constants $c_1>0$ and $c_2<\infty$ such that
\begin{eqnarray}\label{domin.b}
c_1\overline F(x)\le\Pr\{W_n>x\} \le c_2\overline F(x)
\quad\mbox{ for all }x,\ n.
\end{eqnarray}
If, in addition, $F$ is intermediate regularly varying distribution,
then, uniformly in $n$,
\begin{eqnarray}\label{irv.1}
\Pr\{W_n>x\} &\sim& \sum_{i=0}^{n-1} m^i \overline F(m^{i+1}x)
\quad\mbox{ as }x\to\infty.
\end{eqnarray}
In particular,
\begin{eqnarray}\label{irv.2}
\Pr\{W_n>x\} &\sim& \sum_{i=0}^\infty m^i \overline F(m^{i+1}x)
\quad\mbox{ as }x, n\to\infty,
\end{eqnarray}
and
\begin{eqnarray}\label{irv.3}
\Pr\{W>x\} &\sim& \sum_{i=0}^\infty m^i \overline F(m^{i+1}x)
\quad\mbox{ as }x\to\infty.
\end{eqnarray}
\end{theorem}

As follows from the proof of Lemma \ref{l:lower},
$$
\Pr\Bigl\{\max_{i\le Z_k}\xi^{(k)}_i\ge m^{k+1}x\Bigr\}
\sim m^k\overline{F}(m^{k+1}x)
\quad\mbox{ as }x\to\infty
$$
and the summand $m^k\overline{F}(m^{k+1}x)$ in \eqref{irv.1}--\eqref{irv.3}
describes the probability of the existence of a very productive
particle in the $k$-th generation.
We can informally restate \eqref{irv.1}--\eqref{irv.3} as follows
$$
\{W_n>x\}\approx\bigcup_{k=0}^{n-1}
\Bigl\{\max_{i\le Z_k}\xi^{(k)}_i\ge m^{k+1}x\Bigr\}
$$
and
$$
\{W>x\}\approx\bigcup_{k=0}^\infty
\Bigl\{\max_{i\le Z_k}\xi^{(k)}_i\ge m^{k+1}x\Bigr\}.
$$
Moreover, if $\overline{F}(x)$ is regularly varying with index $\alpha<-1$
then, uniformly in $n$,
$$
\Pr\Bigl\{\max_{i\le Z_k}\xi^{(k)}_i\ge m^{k+1}x\Big|W_n>x\Bigr\}\to
\frac{m^{-(\alpha-1)k}}{\sum_{j=0}^{n-1}m^{-(\alpha-1)j}}
\quad\mbox{ as }x\to\infty.
$$
In the limit $n\to\infty$ we get the geometric distribution with
the parameter $m^{-(\alpha-1)}$. Therefore, atypically big values of
the limit $W$ are caused by a very productive particle which lives
in one of the initial generations, and the number of this generation
is random with the geometric distribution mentioned above.

If we assume the second moment of $\xi$ finite then we may relax
the regularity condition on $F$, namely we may consider
distributions which are not necessarily intermediate regularly
varying as was assumed in Theorem \ref{thm:irv}.

\begin{theorem}\label{thm:sq.root}
Let $F$ be dominated varying distribution and
the condition \eqref{cond:mat} hold. If $\E\xi^2<\infty$
and $F$ is $x^\gamma$-insensitive distribution for some
$\gamma>1/2$, then the asymptotics \eqref{irv.1}, \eqref{irv.2}
and \eqref{irv.3} hold.
\end{theorem}

We next turn to the case of Weibull-type offspring distributions.

\begin{theorem}\label{thm:asy.uni.W}
Let $\overline F(x)=e^{-R(x)}$ where $R(x)$ is regularly
varying with index $\beta\in(0,1)$.
Additionally assume that $F\in{\mathcal S}^*$.
Then, for every $\varepsilon>0$,
$$
(1+o(1))\overline F((m+\varepsilon)x)
\le \Pr\{W_n>x\} \le (1+o(1))\overline F((m-\varepsilon)x)
$$
as $x\to\infty$ uniformly in $n$.

If $\beta<\frac{3-\sqrt 5}{2}\approx 0.382$ then
$\Pr\{W_n>x\} \sim \overline F(mx)$
as $x\to\infty$ uniformly in $n$ and
$\Pr\{W>x\} \sim \overline F(mx)$ as $x\to\infty$.

If $\beta<1/2$ and, in addition, for some $c_1<\infty$,
\begin{eqnarray}\label{cond:R.der}
R(k)-R(k-1) &\le& c_1\frac{R(k)}{k},\quad k\ge 1,
\end{eqnarray}
then $\Pr\{W_n>x\} \sim \Pr\{W>x\}\sim\overline F(mx)$
as $x\to\infty$ uniformly in $n$.
\end{theorem}

Let us make a remark on Weibull-type offspring distributions
which are not $\sqrt x$-insensitive. If
$\Pr\{\xi>x\}\sim e^{-x^\beta}$ with some $\beta\in(1/2,1)$, then
\begin{equation}\label{Rem1}
\Pr\{W_n>x\}\ge \exp\Bigl\{-(mx)^\beta
+\frac{\beta^2\sigma_n^2}{2}(mx)^{2\beta-1}(1+o(1))\Bigr\},
\quad n\ge 2,
\end{equation}
and
\begin{equation}\label{Rem2}
\Pr\{W>x\}\ge \exp\Bigl\{-(mx)^\beta
+\frac{\beta^2\sigma^2}{2}(mx)^{2\beta-1}(1+o(1))\Bigr\}.
\end{equation}
Here $\sigma_n^2:=\E(W_n-1)^2$ and $\sigma^2:=\E(W-1)^2$.
These bounds imply that, in contrast to the case $\beta<1/2$,
$\Pr\{W_n>x\}\gg\overline F(mx)$ for all $n\ge 2$.
At the end of Section \ref{sec:lower.bounds}
we give arguments for (\ref{Rem2}).

In Theorem \ref{thm:asy.uni.W} we have, uniformly in $n$,
$$
\{W_n>x\}\approx \{\xi^{(0)}_1>mx\}.
$$
Thus, large values of all $W_n$ are caused by a correspondingly large
first generation.

The importance of initial generations for deviation probabilities
can be explained by the multiplicative structure of supercritical
Galton-Watson processes. As a consequence of this fact, it is
`cheaper' to have some special type of behaviour at the very
beginning of the process. In Theorems \ref{thm:irv} and
\ref{thm:asy.uni.W} we see a quite strong localisation: only finite
number of generations is important. There are some examples in the
literature where a weaker form of the localisation occurs. In
the case of lower deviations which were studied in \cite{FW07,FW09},
the optimal strategy looks as follows: In order to have
$\{Z_n=k_n\}$ with some $k_n=o(m^n)$ every particle in first $a_n$
generations should have exactly $\mu:=\min\{k:\Pr\{\xi=k\}>0\}$
children. (Here we assume, for simplicity, that $\xi\ge 1$.) In all
later generations we let $Z_k$ grow without any restriction, i.e.,
geometrically with the rate $m$. Since we want to get $k_n$ particles
in the $n$-th generation, $a_n$ should satisfy $\mu^a_nm^{n-a_n}\approx k_n$.
Recalling that $k_n=o(m^n)$, we see that the number of generations
with non-typical behaviour tends to infinity. A similar strategy is
behind asymptotics for $\Pr\{W<\varepsilon\}$ as
$\varepsilon\to0$ and behind asymptotics for upper deviations of
processes with polynomial generating functions. This localisation
effect for Galton-Watson processes with vanishing limit,
that is, $Z_n$ conditioned on $\{W<\varepsilon\}$ with $\varepsilon\to0$,
was recently studied by Berestycki, Gantert, M\"orters and Sidorova \cite{BGMS}.
They showed that the genealogical tree coinsides up to a certain generation
with the regular $\mu$-ary tree.

It turns out that this type of optimal strategies is not universal for
supercritical Galton-Watson processes. The next result shows that if
the offspring distribution has only the first power moment,
then large values of $W_n$ and $W$ can be produced by the middle
part of the tree.

\begin{theorem}\label{T10}
Assume that $\E\xi\log\xi<\infty$ and $\overline F(x)$
is regularly varying with index $-1$. Then, uniformly in $n\ge 1$,
\begin{equation}\label{T10.1}
\Pr\{W_n>x\}\sim\sum_{i=0}^{n-1} m^i \overline F(m^{i+1}x)
\sim\frac{1}{m\log m}x^{-1}\int_x^{m^nx}\overline F(u)du
\quad\mbox{as }x\to\infty.
\end{equation}
For the limit $W$ we have
\begin{equation}\label{T10.11}
\Pr\{W>x\}\sim\sum_{i=0}^\infty m^i \overline F(m^{i+1}x)
\sim\frac{1}{m\log m}x^{-1}\int_x^\infty\overline F(u)du
\quad\mbox{as }x\to\infty.
\end{equation}
\end{theorem}

Relation \eqref{T10.11} is a refinement of Theorem 1.4 in \cite{BD1}
where the following was proved: If $\E\{Z_1;Z_1>x\}\sim L(x)$ for
some slowly varying function $L$ satisfying
$\int_1^\infty\frac{L(x)}{x}dx<\infty$, then
$$
\E\{W;W>x\}\sim \frac{1}{m\log m} \int_x^\infty\frac{L(y)}{y}dy.
$$

Noting that $\overline F(x)=o\left(x^{-1}\int_x^\infty\overline F(u)du\right)$,
we conclude from Theorem \ref{T10} that, for every $N\geq1$,
$$
\sum_{i=0}^{N} m^i \overline F(m^{i+1}x)=o\left(\Pr(W>x)\right)
\quad\mbox{as }x\to\infty.
$$
This means that 'big jumps' in any fixed number of generations do
not affect large values of $W$. Furthermore, the main contribution to
$\sum_{i=0}^\infty m^i \overline F(m^{i+1}x)$ (and therefore to $\Pr\{W>x\}$)
comes from indices $i$ such that the ratio
$\frac{\int_{m^ix}^\infty\overline F(u)du}{\int_x^\infty\overline F(u)du}$
is bounded away from $0$ and $1$. For finite values of $n$ we have three different
regimes depending on the relation between $n$ and $x$.
We illustrate them by the following example.

\begin{example}
Assume that $\overline F(x)\sim x^{-1}\log^{-p-1}x$ with some $p>1$. Then
$$
L(x):=\int_x^\infty\overline{F}(y)dy\sim\frac{1}{p}\log^{-p}x
\quad\mbox{as }x\to\infty.
$$
Therefore,
\begin{equation}\label{Ex1}
\Pr\{W>x\}\sim\frac{1}{m\log m}x^{-1}L(x)\sim\frac{1}{pm\log m}x^{-1}\log^{-p}x.
\end{equation}

Consider now finite values of $n$.

First, if $n$ and $x$ are such that $\frac{n}{\log x}\to\infty$,
then, according to \eqref{T10.1},
\begin{eqnarray*}
\Pr\{W_n>x\} &\sim& \frac{1}{m\log m}x^{-1}\left(L(x)-L(m^nx)\right)\\
&\sim& \frac{1}{m\log m}x^{-1}L(x)\sim\frac{1}{pm\log m}x^{-1}\log^{-p}x.
\end{eqnarray*}
Comparing this with \eqref{Ex1}, we see that asyptotics of $\Pr\{W_n>x\}$
and $\Pr\{W>x\}$ are equal in this case.

Second, if $n$ and $x$ are such that $\frac{n}{\log x}\to t\in(0,\infty)$,
then
$$
L(m^nx)\sim\frac{1}{p}\left(\log x+n\log m\right)^{-p}
\sim \frac{1}{p}\log^{-p}x\left(1+t\log m\right)^{-p}.
$$
Consequently,
$$
\Pr\{W_n>x\}\sim \frac{1}{pm\log m}x^{-1}\log^{-p}x
\left(1-\left(1+t\log m\right)^{-p}\right).
$$
Here we see that $\Pr\{W_n>x\}$ and $\Pr\{W>x\}$ are still
of the same order, but the constants are different.

Third, if $\frac{n}{\log x}\to0$ then, noting
that $\log y\sim \log x$ uniformly in $y\in[x,m^nx]$, we have
\begin{eqnarray*}
\Pr\{W_n>x\} &\sim& \frac{1}{m\log m}x^{-1}\int_x^{m^nx}\frac{dy}{y\log^{p+1}y}\\
&\sim& \frac{1}{m\log m}\frac{1}{x\log^{p+1}x}\int_x^{m^nx}\frac{dy}{y}
\sim n\overline{F}(mx).
\end{eqnarray*}
Therefore, $\Pr\{W_n>x\}$ is much smaller than $\Pr\{W>x\}$ for these values
of $n$.
\end{example}

The problem of describing tail asymptotics for supercritical
Galton--Watson process is closely related to the problem
of tail behaviour for randomly stopped sum $S_\tau$ where
the random number $\tau$ of summands has the same distribution
as the summands $\xi$'s have.
For random sums, the only case well studied is the case where
the distribution tail of $\tau$ is much lighter than that
of $\xi$, see \cite{DFK}; in this case the typical answer
is $\Pr\{S_\tau>x\}\sim\E\tau\Pr\{\xi>x\}$ as $x\to\infty$.
The present study may be considered
as a step towards general problem for randomly stopped sums
where the tails of the stopping time $\tau$ and of the summand
$\xi$ are comparable.

The rest of the paper is organised as follows.
Section \ref{sec:sums} is devoted to related upper bounds
for the distribution tails of sums with zero drift in the
of large deviation zone. Later on in Section \ref{sec:upper}
they serve for deriving upper bounds for $\Pr\{W_n>x\}$;
more precisely, we reduce the problem of finding
the asymptotic behaviour of $\Pr\{W_n>x\}$ to that for
$\Pr\{W_N>x\}$ with some fixed $N$. Also, upper bounds
of Section \ref{sec:sums} help to compute asymptotics
for $\Pr\{W_N>x\}$ for every fixed $N$. Lower bounds for the
distribution tail of the number of descendants in
the $n$th generation are given
in Section \ref{sec:lower.bounds}. In Section \ref{sec:proofs}
we provide final proofs of Theorems \ref{thm:irv},
\ref{thm:sq.root} and  \ref{thm:asy.uni.W}.
Finally, for Theorem \ref{T10} where only the first moment
is finite, our approach based on describing and computing
the most likely events leading to large deviations of $W_n$
doesn't work. Here we propose an analytic method adapted from
\cite{NV}, see Section \ref{sec:first.moment}.

\section{Preliminary results for sums}
\label{sec:sums}

We repeatedly make use of the following result which is a
version of Theorem 2(i) in \cite{DFK} with exactly
the same proof. In what follows
$\eta_1$, $\eta_2$, \ldots are independent random variables with
common distribution $G$ and $T_n:=\eta_1+\ldots+\eta_n$.

\begin{proposition}\label{S.star}
Let the distribution $G$ have negative mean $a:=\E\eta_1<0$.
If $G\in{\mathcal S}^*$ then
$$
\Pr\{T_n>x\}\le(1+o(1)) n\overline G(x)
$$
as $x\to\infty$ uniformly in $n$.
\end{proposition}

The latter proposition helps to deduce exact asymptotics for
$\Pr\{T_n>x\}$ in the case of zero mean if $x/n>c>0$.
If $x=o(n)$ then Proposition \ref{S.star} is not useful for
estimation of $\Pr\{T_n>x\}$ in the case of zero mean.
So, in the following two propositions we derive rough upper bounds
for the large deviation probabilities for sums with zero mean;
these rough bounds will be appropriate for our purposes.
The first proposition is devoted to distributions of regularly
varying type while the second one is devoted to Weibullian
type distributions. Deriving rather rough bounds,
we relax conditions on distribution of jumps
comparing to the asymptotic results of
\cite[Theorems 8.1 and 8.3]{DDS}
and \cite[Theorems 3.1.1, 4.1.2 and 5.2.1]{BBs}.

\begin{proposition}\label{thm:upper.bound.for.sum.R}
Let $\E\eta_1=0$, $\E\{\eta_1^2;\eta_1\le 0\}<\infty$ and
$G$ be a dominated varying distribution.

If, for some $\delta\in(0,1)$,
\begin{eqnarray}\label{sandw.1}
\E\{\eta^{1+\delta};\eta>0\}<\infty,
\end{eqnarray}
then, for every $\delta'\in(0,\delta)$, there exists $c<\infty$
such that $\Pr\{T_n>x\} \le cn\overline G(x)$
for all $x>0$ and $n\le x^{1+\delta'}$.

If, for some $\delta>0$,
\begin{eqnarray}\label{sandw}
\E\{\eta^{2+\delta};\eta>0\}<\infty,
\end{eqnarray}
then there exists $c<\infty$ such that $\Pr\{T_n>x\} \le cn\overline G(x)$
for all $x>0$ and $n\le x^2/c\log x$
{\rm(}or equivalently, $x\ge c\sqrt{n\log n}${\rm)}.
\end{proposition}

\begin{proof}
Let $R(x)$ be the hazard function for $G$, that is,
$\overline G(x)=e^{-R(x)}$. First prove that dominated
variation yields, for some $C<\infty$, the upper bound
\begin{eqnarray}\label{sandw.k}
R(x) &\le& C+C\log x,\quad x\ge 1.
\end{eqnarray}
Indeed, there exists $c<\infty$ such that
$\overline G(x/2)\le e^c\overline G(x)$ for all $x$.
Equivalently, $R(x/2)\ge R(x)-c$ which implies
$R(x2^{-n})\ge R(x)-cn$. For $n(x):=[\log_2 x]+1$ we get
\begin{eqnarray*}
R(1) &\ge& R(x2^{-n(x)})\\
&\ge& R(x)-cn(x)\ge R(x)-c\log_2 x-c
\end{eqnarray*}
and the upper bound \eqref{sandw.k} follows.

For every $y<x$, we may estimate
the tail distribution of the sum as follows:
\begin{eqnarray}\label{Tn.le}
\Pr\{T_n>x\} &\le& \Pr\{T_n>x,\eta_k>y
\mbox{ for some }k\le n\}
+\Pr\{T_n>x,\eta_k\le y\mbox{ for all }k\le n\}\nonumber\\
&\le& n\overline G(y)+e^{-\lambda x}
(\E\{e^{\lambda\eta_1};\eta_1\le y\})^n,
\end{eqnarray}
for every $\lambda>0$, by the exponential Chebyshev inequality.
Fix an $\varepsilon\in(0,1)$. Take $y:=\varepsilon x$ and
$\lambda:=2R(x)/x$.
Then $e^{-\lambda x}=\overline G(x)e^{-R(x)}$ and
\begin{eqnarray*}
\Pr\{T_n>x\} &\le& n\overline G(y)+\overline G(x)e^{-R(x)}
(\E\{e^{\lambda\eta_1};\eta\le y\})^n.
\end{eqnarray*}
Let us estimate the latter truncated exponential moment:
\begin{eqnarray}\label{estimate.for.R}
\E\{e^{\lambda\eta_1};\eta\le y\}
&=& \E\{e^{\lambda\eta_1};\eta_1\le 1/\lambda\}
+\E\{e^{\lambda\eta_1};1/\lambda<\eta_1\le y\}.
\end{eqnarray}
Since $e^u\le 1+u+2u^2$ for all $u\le 1$,
\begin{eqnarray}\label{estimate.for.2.R.gen}
\E\{e^{\lambda\eta_1};\eta_1\le 1/\lambda\}
&\le& 1+\lambda\E\{\eta_1;\eta_1\le 1/\lambda\}
+2\lambda^2\E\{\eta_1^2;\eta_1\le 1/\lambda\}\nonumber\\
&\le& 1+2\lambda^2\E\{\eta_1^2;\eta_1\le 1/\lambda\},
\end{eqnarray}
owing to the mean zero for $\eta_1$.

Consider the case of finite second moment where we get
\begin{eqnarray}\label{estimate.for.2.R}
\E\{e^{\lambda\eta_1};\eta_1\le 1/\lambda\} &\le& 1+c_1\lambda^2.
\end{eqnarray}
Further,
\begin{eqnarray*}
\E\{e^{\lambda\eta_1};1/\lambda<\eta_1\le y\}
&\le& e^{\lambda y} \overline G(1/\lambda)\\
&\le& e^{\lambda y} \E\{\eta^{2+\delta};\eta>0\}\lambda^{2+\delta},
\end{eqnarray*}
by the condition \eqref{sandw} and the Chebyshev inequality.
Choose $\varepsilon>0$ so small that $\varepsilon C<\delta/4$.
Then the upper bound \eqref{sandw.k} yields, for some $c_2<\infty$,
$$
e^{\lambda y}=e^{\varepsilon x 2R(x)/x}\le c_2x^{\delta/2}
$$
and consequently
\begin{eqnarray}\label{estimate.for.3.R}
\E\{e^{\lambda\eta_1};1/\lambda<\eta_1\le y\} &\le& c_3 \lambda^2.
\end{eqnarray}
Together with \eqref{estimate.for.2.R} it implies that
\begin{eqnarray*}
\E\{e^{\lambda\eta_1};\eta_1\le y\}
&\le& 1+c_4R^2(x)/x^2 \le e^{c_4R^2(x)/x^2},
\end{eqnarray*}
for some $c_4<\infty$. Hence,
\begin{eqnarray*}
\Pr\{T_n>x\} &\le& n\overline G(y)
+\overline G(x)e^{-R(x)} e^{c_4nR^2(x)/x^2}\\
&\le& n\overline G(y)
+\overline G(x)e^{-R(x)+R(x)(c_5n\log x/x^2)},
\end{eqnarray*}
for some $c_5<\infty$, due to \eqref{sandw.k}.
So, in the case of finite $2+\delta$ moment, the proposition conclusion
follows for $n\le x^2/c_5\log x$ if we take into account
\eqref{class.D}.

In the case where the condition \eqref{sandw.1} only holds,
$$
\E\{\eta_1^2;\eta_1\le 1/\lambda\}
\le \E\{\eta_1^2;\eta_1\le 0\}
+\E\{\eta_1^{1+\delta};\eta_1>0\}/\lambda^{1-\delta}
$$
and we deduce from the estimate \eqref{estimate.for.2.R.gen} that
\begin{eqnarray*}
\E\{e^{\lambda\eta_1};\eta_1\le 1/\lambda\}
&\le& 1+c_6\lambda^{1+\delta}.
\end{eqnarray*}
Similar to \eqref{estimate.for.3.R},
\begin{eqnarray*}
\E\{e^{\lambda\eta_1};1/\lambda\le \eta_1\le y\}
&\le& c_7/x^{1+\delta'}.
\end{eqnarray*}
by the condition \eqref{sandw.1}. Then
\begin{eqnarray*}
\E\{e^{\lambda\eta_1};\eta_1\le y\}
&\le& 1+c_8/x^{1+\delta'} \le e^{c_8/x^{1+\delta'}},
\end{eqnarray*}
because $R(x)\le c_9\log x$ by \eqref{sandw.k}. Hence,
\begin{eqnarray*}
\Pr\{T_n>x\} &\le& n\overline G(y)
+\overline G(x)e^{-R(x)} e^{c_8n/x^{1+\delta'}},
\end{eqnarray*}
and the case of finite first moment follows.
\end{proof}

\begin{proposition}\label{thm:upper.bound.for.sum.W}
Let the distribution $G$ have mean zero, $\E\eta_1=0$,
and all moments finite, $\E|\eta_1|^k<\infty$, $k=1$, $2$, \ldots.
Let $R(x)$ be the hazard function for $G$, that is,
$\overline G(x)=e^{-R(x)}$.
Suppose, for every $\varepsilon>0$, there exists $x_0$ such that
\begin{eqnarray}\label{cond.non.increasing}
R(x)/x &\le& (1+\varepsilon)R(z)/z\quad\mbox{ for all }x\ge z\ge x_0.
\end{eqnarray}
Then, for every $0<\varepsilon<1$, there exists
a $c=c(\varepsilon)<\infty$ such that
$$
\Pr\{T_n>x\} \le (n+1)\overline G(y)
$$
for all $x>0$, $y\le(1-\varepsilon)x$ and $n$ such that
$nR(y)/x^2\le 1/c$.
\end{proposition}

\begin{proof}
Take $\lambda:=(1+\varepsilon)R(y)/x$. Then
$e^{-\lambda x}=e^{-(1+\varepsilon)R(y)}$.

By the condition \eqref{cond.non.increasing},
$$
\lambda z = (1+\varepsilon)\frac{R(y)}{y}\frac{y}{x}z
\le (1-\varepsilon^2)\frac{R(y)}{y}z
\le (1-\varepsilon^2/2)R(z)
$$
for all $z\le y$ sufficiently large. Therefore,
\begin{eqnarray*}
\E\{e^{\lambda\eta_1};1/\lambda<\eta_1\le y\}
&\le& \E\{e^{(1-\varepsilon^2/2)R(\eta_1)};1/\lambda<\eta_1\}\\
&\le& -\int_{1/\lambda}^\infty e^{(1-\varepsilon^2/2)R(z)}de^{-R(z)}\\
&=& \int_{1/\lambda)}^\infty e^{-\varepsilon^2 R(z)/2}dR(z)\\
&=& 2e^{-\varepsilon^2 R(1/\lambda)/2}/\varepsilon^2.
\end{eqnarray*}
Taking into account that, for every $\alpha>0$,
$e^{-R(x)}=o(1/x^\alpha)$ as $x\to\infty$, we get
\begin{eqnarray}\label{estimate.for.3}
\E\{e^{\lambda\eta_1};1/\lambda<\eta_1\le y\}
&=& o(\lambda^2)\quad\mbox{ as }y\to\infty.
\end{eqnarray}
Substituting \eqref{estimate.for.2.R} and \eqref{estimate.for.3}
into \eqref{estimate.for.R} we obtain the following inequality
\begin{eqnarray*}
\E\{e^{\lambda\eta_1};\eta_1\le y\}
&\le& 1+c R^2(y)/x^2 \le e^{c R^2(y)/x^2},
\end{eqnarray*}
for some $c<\infty$. Hence,
\begin{eqnarray*}
\Pr\{T_n>x\} &\le& n\overline G(y)
+e^{-(1+\varepsilon)R(y)}e^{c_1n R^2(y)/x^2}\\
&\le& n\overline G(y)+e^{-R(y)}=(n+1)\overline G(y)
\end{eqnarray*}
in the range where $c_1nR(y)/x^2\le\varepsilon$
and the proof of the desired upper bound is complete.
\end{proof}

In the proof above the distribution $G$ restricted to
$(-\infty,1/\lambda]$ comes into the upper bound through its
second moment only. The tail of $G$ influences the upper
bound though its values right to the point $1/\lambda$.
Having this observation in mind, we formulate
the following uniform version of the previous proposition
for a family of distributions whose tails are
ultimately dominated by that of $G$.

\begin{corollary}\label{cor:upper.bound.for.sum.W}
Let all the conditions of Proposition
\ref{thm:upper.bound.for.sum.W} be fulfilled.
Let $G^{(v)}$ be a family of distributions
depending on some parameter $v\in V$ such that,
for some $x_1$, $\overline{G^{(v)}}(x)\le\overline G(x)$
for all $x>x_1$ and $v\in V$. Let every $G^{(v)}$
have mean zero and let all the second moments be bounded.
Then, for every $0<\varepsilon<1$, there exists
a $c=c(\varepsilon)<\infty$ such that
$$
\overline{(G^{(v)})^{*n}}(x) \le (n+1)\overline{G^{(v)}}(y)
$$
for all $v\in V$, $x>0$, $y\le(1-\varepsilon)x$ and $n$
such that $nR(y)/x^2\le 1/c$.
\end{corollary}

\section{Lower bounds}
\label{sec:lower.bounds}

\begin{lemma}\label{l:lower}
Let $\E\xi\log\xi<\infty$. Then, for every $\varepsilon>0$,
$$
\Pr\{W_n>x\} \ge (1+o(1))\sum_{i=0}^{n-1} m^i
\overline F(m^{i+1}(1+\varepsilon)x)
$$
as $x\to\infty$ uniformly in $n\ge 1$.
\end{lemma}

\begin{proof}
Consider the following decreasing sequence of events
$$
B_k(x):=\{Z_j\le m^jx\mbox{ for all }j\le k\}.
$$
Since $Z_j/m^j\to W$ a.s. as $j\to\infty$,
\begin{eqnarray}\label{bound.for.Bk}
\inf_{k\ge 1}\Pr\{B_k(x)\} &\to& 1\quad\mbox{as }x\to\infty.
\end{eqnarray}
The events
$$
A_k(x):=
\{B_k(x),\xi_i^{(k)}>m^{k+1}(1+\varepsilon)x\mbox{ for some }i\le Z_k\}
$$
are disjoint which implies the lower bound
\begin{eqnarray}\label{Wn.xi}
\Pr\{W_n>x\} &\ge& \sum_{k=0}^{n-1}
\Pr\{Z_n>m^nx\mid A_k(x)\}\Pr\{A_k(x)\}.
\end{eqnarray}
First we estimate the probability
\begin{eqnarray*}
\Pr\{A_k(x)\} &=& \sum_{j=0}^{m^kx} \Pr\{B_k(x),Z_k=j\}
\Pr\{\xi_i^{(k)}>m^{k+1}(1+\varepsilon)x\mbox{ for some }i\le j\}\\
&=& \sum_{j=0}^{m^kx} \Pr\{B_k(x),Z_k=j\}
\bigl(1-\bigl(1-\overline F(m^{k+1}(1+\varepsilon)x)\bigr)^j\bigr).
\end{eqnarray*}
Since $\E\xi\log\xi<\infty$, by the Chebyshev inequality
$$
\Pr\{\xi>m^{k+1}(1+\varepsilon)x\} \le
\frac{\E\xi\log\xi}{m^{k+1}x\log x}=o(1/m^kx)
\quad\mbox{as }x\to\infty\mbox{ uniformly in }k.
$$
Hence,
$$
\bigl(1-\overline F(m^{k+1}(1+\varepsilon)x)\bigr)^j
=1-j\overline F(m^{k+1}(1+\varepsilon)x)(1+o(1))
$$
as $x\to\infty$ uniformly in $k\ge 0$ and $j\le m^kx$.
Therefore,
\begin{eqnarray*}
\Pr\{A_k(x)\} &=& (1+o(1))\overline F(m^{k+1}(1+\varepsilon)x)
\sum_{j=0}^{m^kx} j\Pr\{B_k(x),Z_k=j\}
\end{eqnarray*}
as $x\to\infty$ uniformly in $k\ge 0$.
The Kesten--Stigum theorem (see, e.g \cite[Theorem 2.1]{A1983})
states, in particular, that $\E\xi\log\xi<\infty$
if and only if the family of
random variables $\{W_n, n\ge 0\}$ is uniformly integrable.
Therefore, it follows from \eqref{bound.for.Bk} that
$$
\E\{W_k;B_k(x),W_k\le x\}\to 1\quad\mbox{as }x\to\infty
\mbox{ uniformly in }k.
$$
By this reason,
\begin{eqnarray*}
\sum_{j=0}^{m^kx} j\Pr\{B_k(x),Z_k=j\}
&=& \E\{Z_k;B_k(x),Z_k\le m^kx\}\\
&=& m^k\E\{W_k;B_k(x),W_k\le x\}\sim m^k
\end{eqnarray*}
as $x\to\infty$ uniformly in $k\ge 0$. Thus,  uniformly in $k\ge 0$,
\begin{eqnarray}\label{Wn.xi.1}
\Pr\{A_k(x)\} &=& (1+o(1))m^k\overline F(m^{k+1}(1+\varepsilon)x)
\quad\mbox{as }x\to\infty.
\end{eqnarray}

Second we prove that
\begin{eqnarray}\label{Wn.xi.2}
\inf_{n\ge 1,\ k\le n-1}\Pr\{Z_n>m^n(1+\varepsilon)x\mid A_k(x)\} &\to& 1
\quad\mbox{as }x\to\infty.
\end{eqnarray}
Indeed, by the Markov property,
\begin{eqnarray*}
\Pr\{Z_n>m^nx\mid A_k(x)\} &\ge&
\Pr\Bigl\{\sum_{j=1}^{m^{k+1}(1+\varepsilon)x}
Z_{n-k-1,j}>m^nx\Bigr\}\\
&=& \Pr\Bigl\{\sum_{j=1}^{m^{k+1}(1+\varepsilon)x}
W_{n-k-1,j}>m^{k+1}x\Bigr\},
\end{eqnarray*}
where $Z_{n-k-1,j}$ are independent copies of $Z_{n-k-1}$
and $W_{n-k-1,j}$ are independent copies of $W_{n-k-1}$.
Since the family $\{W_n\}$ is uniformly integrable and
$\E W_n=1$ for every $n$, we may apply the law of large numbers
which ensures that
$$
\frac{1}{m^{k+1}x}\sum_{j=1}^{m^{k+1}(1+\varepsilon)x}W_{n-k-1,j}
\stackrel{p}{\to} (1+\varepsilon)\E W_{n-k-1}=1+\varepsilon
$$
as $x\to\infty$ uniformly in $n\ge 1$ and $k\le n-1$.
Therefore,
\begin{eqnarray*}
\Pr\Bigl\{\sum_{j=1}^{m^{k+1}(1+\varepsilon)x}
W_{n-k-1,j}>m^{k+1}x\Bigr\} &\to& 1,
\end{eqnarray*}
which justifies the convergence \eqref{Wn.xi.2}.
Substituting \eqref{Wn.xi.1} and \eqref{Wn.xi.2} into
\eqref{Wn.xi}, we deduce the desired lower bound uniform in $n$.
\end{proof}

\begin{lemma}\label{l:lower.2}
Let the distribution $F$ have the second moment finite,
$\sigma^2:={\rm Var}\xi_1<\infty$. Then, for every $A>0$,
\begin{eqnarray}\label{lower.11}
\Pr\{W_n>x\} &\ge& \Bigl(1-\frac{\sigma^2}{(m^2-m)A^2}+o(1)\Bigr)
\sum_{i=0}^{n-1} m^i
\overline F(m^{i+1}x+A\sqrt{m^{i+1}x})
\end{eqnarray}
as $x\to\infty$ uniformly in $n$.

In particular, if additionally the distribution $F$
is $\sqrt x$-insensitive, then
\begin{eqnarray}\label{lower.22}
\Pr\{W_n>x\} &\ge& (1+o(1))\sum_{i=0}^{n-1} m^i
\overline F(m^{i+1}x)\quad\mbox{ as }x\to\infty
\mbox{ uniformly in }n.
\end{eqnarray}
\end{lemma}

\begin{proof}
Let events $B_k(x)$ be defined as above and
$$
A_k(x):=
\{B_k(x),\xi_i^{(k)}>m^{k+1}x+A\sqrt{m^{k+1}x}
\mbox{ for some }i\le Z_k\}
$$
which again are disjoint which implies the lower bound \eqref{Wn.xi}.
The same calculations as in the previous proof lead to
the relation, uniformly in $k\ge 0$,
\begin{eqnarray}\label{Wn.xi.squre.1}
\Pr\{A_k(x)\} &=&
(1+o(1))m^k\overline F(m^{k+1}x+A\sqrt{m^{k+1}x})
\quad\mbox{as }x\to\infty.
\end{eqnarray}
Then it remains to prove that
\begin{eqnarray}\label{Wn.xi.squre.2}
\liminf_{x\to\infty}\inf_{n\ge 1,\ k\le n-1}
\Pr\{Z_n>m^nx\mid A_k(x)\} &\ge& 1-\frac{\sigma^2}{(m^2-m)A^2}.
\end{eqnarray}
Indeed,
\begin{eqnarray*}
\Pr\{Z_n>m^nx\mid A_k(x)\} &\ge&
\Pr\Bigl\{\sum_{j=1}^{m^{k+1}x+A\sqrt{m^{k+1}x}}
Z_{n-k-1,j}>m^nx\Bigr\}\\
&=& \Pr\Bigl\{\sum_{j=1}^{m^{k+1}x+A\sqrt{m^{k+1}x}}
W_{n-k-1,j}>m^{k+1}x\Bigr\}\\
&=& \Pr\Bigl\{\sum_{j=1}^{m^{k+1}x+A\sqrt{m^{k+1}x}}
(W_{n-k-1,j}-1)>-A\sqrt{m^{k+1}x}\Bigr\},
\end{eqnarray*}
since $\E W_n=1$.
Applying the Chebyshev's inequality, we deduce
\begin{eqnarray*}
\Pr\{Z_n>m^nx\mid A_k(x)\} &\ge&
1-\frac{{\rm Var}W_{n-k-1}}{A^2}
\frac{m^{k+1}x+A\sqrt{m^{k+1}x}}{m^{k+1}x}\\
&=& 1-\frac{{\rm Var}W_{n-k-1}}{A^2}(1+o(1))
\end{eqnarray*}
as $x\to\infty$ uniformly in $n\ge1$ and $k\le n-1$.
As calculated in \cite[Theorem 1.5.1]{H},
\begin{eqnarray*}
{\rm Var}W_n &=& \frac{\sigma^2 (1-m^{-n})}{m^2-m}
\uparrow\frac{\sigma^2}{m^2-m}={\rm Var}W
\quad\mbox{ as }n\to\infty,
\end{eqnarray*}
which completes the proof of \eqref{Wn.xi.squre.2}.
Substituting \eqref{Wn.xi.squre.1} and \eqref{Wn.xi.squre.2}
into \eqref{Wn.xi} we deduce the lower bound \eqref{lower.11}.

If $F$ is $\sqrt x$-insensitive, then letting $A\to\infty$
we conclude the second lower bound of the lemma.
\end{proof}

As clearly seen from the proof of Lemma \ref{l:lower.2},
in the case of Weibull distribution with parameter $\beta\in(1/2,1)$
the tail of $W_n$ is definitely heavier than $\overline F(mx)$.
Now let us explain why more accurate lower bound (\ref{Rem2})
given in Introduction holds. Recalling that
$$
W=\frac{1}{m}\sum_{i=1}^\xi W^{(i)},
$$
where $W^{(i)}$ are independent copies of $W$ which
don't depend on $\xi$, we derive
\begin{align}\label{Rem3}
\nonumber
\Pr\{W>x\}
&\ge \Pr\biggl\{\sum_{i=1}^\xi W^{(i)}>mx;\,\xi\geq N_x\biggr\}\\
&\ge \Pr\{\xi>N_x\}\Pr\biggl\{\sum_{i=1}^{N_x} W^{(i)}>mx\biggr\},
\end{align}
where $N_x:=[mx-z(mx)^\beta]$, $z>0$. It is easy to see that
\begin{align}\label{Rem4}
\nonumber
\Pr\{\xi>N_x\}&\sim e^{-(mx-z(mx)^\beta)^\beta}\\
&= e^{-(mx)^\beta+\beta z(mx)^{2\beta-1}+O(x^{3\beta-2})}.
\end{align}
In view of log-scaled asymptotics for $\Pr\{W>x\}$
(see the first assertion of Theorem \ref{thm:asy.uni.W}),
$\E e^{(1-\varepsilon)m^\beta W^\beta}<\infty$
for every $\varepsilon>0$.
Moreover, $x^{2\beta}\ll N_x(x^\beta)^\beta$.
Consequently, we may apply Nagaev's theorem
\cite[Theorem 3]{N1965}:
\begin{align}\label{Rem5}
\nonumber
\Pr\biggl\{\sum_{i=1}^{N_x} W^{(i)}>mx\biggr\}&\ge
\Pr\biggl\{\sum_{i=1}^{N_x} (W^{(i)}-1)>z(mx)^\beta\biggr\}\\
&=\exp\Bigl\{-\frac{z^2}{2\sigma^2}(mx)^{2\beta-1}(1+o(1))\Bigr\}.
\end{align}
Combining \eqref{Rem3}--\eqref{Rem5}, we get
$$
\Pr\{W>x\}\ge
\exp\bigl\{-(mx)^\beta+\bigl(\beta z-z^2/2\sigma^2\bigr)(mx)^{2\beta-1}(1+o(1))\bigr\}.
$$
Maximizing $\beta z-z^2/2\sigma^2$, we obtain \eqref{Rem2}.

\section{Upper bounds: a reduction to a finite time horizon}
\label{sec:upper}

\begin{lemma}\label{l:upper}
Let the distribution $F$ be dominated varying and satisfy
the condition \eqref{cond:mat}. Then, for every $\varepsilon>0$,
there exists an $N$ such that, for all $n>N$ and for all
sufficiently large $x$,
$$
\Pr\{W_n>x\} \le (1+\varepsilon)\Pr\{W_N>(1-\varepsilon)x\}.
$$
\end{lemma}

\begin{proof}
In order to derive this upper bound we write, for $z<y$,
\begin{eqnarray}\label{est.up}
\Pr\Bigl\{\sum_{i=1}^{Z_{n-1}}\xi_i>my\Bigr\}
&\le& \Pr\{Z_{n-1}>z\}
+\Pr\Bigl\{\sum_{i=1}^{Z_{n-1}}\xi_i>my;Z_{n-1}\le z\Bigr\},
\end{eqnarray}
where the $\xi$'s are independent of $Z_{n-1}$.
It follows from Proposition \ref{thm:upper.bound.for.sum.R}
(under the condition \eqref{sandw.1}) for sums with zero mean,
$\eta_i=\xi_i-m$, that, for some $c<\infty$,
\begin{eqnarray*}
\Pr\Bigl\{\sum_{i=1}^k \xi_i>my\Bigr\} &=&
\Pr\Bigl\{\sum_{i=1}^k (\xi_i-m)>m(y-k)\Bigr\}\\
&\le& ck\overline F(m(y-k))\quad\mbox{ for all }k\le z,
\end{eqnarray*}
provided $z\le (y-z)^{1+\delta/2}$. Therefore,
\begin{eqnarray}\label{for.the}
\Pr\Bigl\{\sum_{i=1}^{Z_{n-1}}\xi_i>my;Z_{n-1}\le z\Bigr\}
&=& \sum_{k=1}^z \Pr\{Z_{n-1}=k\}\Pr\Bigl\{\sum_{i=1}^k \xi_i>my\Bigr\}\nonumber\\
&\le& c\sum_{k=1}^z \Pr\{Z_{n-1}=k\}k\overline F(m(y-k))\\
&\le& c\E Z_{n-1}\overline F(m(y-z))\nonumber\\
&=& cm^{n-1}\overline F(m(y-z)).\nonumber
\end{eqnarray}
Substituting this into \eqref{est.up} with $y=m^{n-1}x$
and $z=m^{n-1}(x-x_n)$ we obtain
\begin{eqnarray*}
\Pr\{W_n>x\} &\le& \Pr\{W_{n-1}>x-x_n\}+cm^{n-1}\overline F(m^n x_n),
\end{eqnarray*}
provided $x-x_n\le m^{(n-1)\delta/2}x_n^{1+\delta/2}$.
Iterating this upper bound $n-N$ times, we arrive at
the following inequality:
\begin{eqnarray}\label{th.N.pre}
\Pr\{W_n>x\} &\le& \Pr\{W_N>x-x_n-\ldots-x_{N+1}\}
+c\sum_{k=N+1}^n m^{k-1}\overline F(m^k x_k),
\end{eqnarray}
provided $x\le m^{(k-1)\delta/2}x_k^{1+\delta/2}$ for all $k$.
Take decreasing sequence $x_k=x/k^2$.
Choose $N$ so large that $m^{(k-1)\delta/2}\ge k^{2+\delta}$
for all $k\ge N+1$.
Then \eqref{th.N.pre} holds for all $n\ge N+1$ and we have
\begin{eqnarray*}
\Pr\{W_n>x\} &\le& \Pr\{W_N>(1-1/(N+1)^2-\ldots-1/n^2)x\}
+c\sum_{k=N+1}^n m^{k-1}\overline F(m^k x/k^2).
\end{eqnarray*}
Choose $N$ so large that additionally
$\sum_{k=N}^\infty 1/k^2\le \varepsilon$. Then
\begin{eqnarray*}
\Pr\{W_n>x\} &\le& \Pr\{W_N>(1-\varepsilon)x\}
+c\sum_{k=N+1}^n m^{k-1}\overline F(m^k x/k^2).
\end{eqnarray*}
Owing to the condition \eqref{cond:mat},
\begin{eqnarray*}
\sum_{k=N+1}^n m^{k-1}\overline F(m^k x/k^2)
&\le& c_1\overline F(mx)\sum_{k=N+1}^\infty
\frac{m^{k-1}}{(m^{k-1}/k^2)^{1+\delta}}.
\end{eqnarray*}
Now we may increase $N$ so that
\begin{eqnarray*}
c\sum_{k=N+1}^n m^{k-1}\overline F(m^k x/k^2)
&\le& \varepsilon\overline F(mx)/3,
\end{eqnarray*}
which implies
\begin{eqnarray*}
\Pr\{W_n>x\} &\le& \Pr\{W_N>(1-\varepsilon)x\}
+\varepsilon\overline F(mx)/2.
\end{eqnarray*}
Applying here Lemma \ref{l:lower}, we deduce
$\overline F(mx)\le(1+o(1))\Pr\{W_N>(1-\varepsilon)x\}$
as $x\to\infty$, so
\begin{eqnarray*}
\Pr\{W_n>x\} &\le& (1+\varepsilon)\Pr\{W_N>(1-\varepsilon)x\}
\end{eqnarray*}
for all sufficiently large $x$, and the proof is complete.
\end{proof}

The calculations above imply the following

\begin{corollary}\label{cor:upper.D}
Let the distribution $F$ be dominated varying and satisfy
the condition \eqref{cond:mat}. Then there exists a constant
$c<\infty$ such that $\Pr\{W_n>x\} \le c\overline F(x)$
for all $n$ and $x$.
\end{corollary}

For dominated varying distributions it is possible to obtain
more accurate bound which will be of use for wider class of
distributions than intermediate regularly varying.
We do it in the next lemma where the bound provided by
the previous corollary serves as the first step preliminary bound.

\begin{lemma}\label{l:upper.sq.N}
Let $\E\xi^2<\infty$, the distribution $F$ be dominated varying
and satisfy the condition \eqref{cond:mat}. Then, for every
$\gamma>1/2$ and $\varepsilon>0$, there exists an $N$ such that,
for all $n>N$ and for all sufficiently large $x$,
$$
\Pr\{W_n>x\} \le (1+\varepsilon)\Pr\{W_N>x-x^\gamma\}.
$$
\end{lemma}

\begin{proof}
Here we need more accurate upper bounds
based on \eqref{for.the}. Take $\delta\in(1/\gamma-1,1)$.
First note that, as follows from Proposition
\ref{thm:upper.bound.for.sum.R} under the condition
\eqref{sandw.1} (which is fulfilled because
$\E\xi^2<\infty$), the bound \eqref{for.the} now holds within a
larger time range where $z\le (y-z)^{1+\delta}$. For those $z$,
\begin{eqnarray*}
\Pr\Bigl\{\sum_{i=1}^{Z_{n-1}}\xi_i>my;Z_{n-1}\le z\Bigr\}
&\le& c\biggl(\sum_{k=1}^{z/2}+\sum_{k=z/2}^z\biggr)
\Pr\{Z_{n-1}=k\}k\overline F(m(y-k))\\
&=:& c(\Sigma_1+\Sigma_2).
\end{eqnarray*}
We have
\begin{eqnarray*}
\Sigma_1 &\le& \overline F(m(y-z/2))\sum_{k=1}^{y/2}\Pr\{Z_{n-1}=k\}k\\
&\le& \E Z_{n-1}\overline F(my/2)
\le c_1m^{n-1}\overline F(my),
\end{eqnarray*}
for some $c_1<\infty$, by dominated variation of $F$. Further,
\begin{eqnarray*}
\Sigma_2 &\le& \Pr\{Z_{n-1}>z/2\}z\overline F(m(y-z))\\
&\le& c_2\overline F(z/2m^{n-1})z\overline F(m(y-z))\\
&\le& c_2c_1\overline F(z/m^{n-1})z\overline F(m(y-z)),
\end{eqnarray*}
by Corollary \ref{cor:upper.D} and dominated variation of $F$.
Collecting bounds for $\Sigma_1$ and $\Sigma_2$ with $y=m^{n-1}x$
and $z=m^{n-1}(x-x_n)$, we obtain from \eqref{est.up} that
\begin{eqnarray*}
\Pr\{W_n>x\} &\le& \Pr\{W_{n-1}>x-x_n\}
+c_1m^{n-1}\overline F(m^nx)
+c_3\overline F(x-x_n)m^{n-1}x\overline F(m^nx_n)
\end{eqnarray*}
provided $x-x_n\le m^{(n-1)\delta}x_n^{1+\delta}$.
Iterating this upper bound $n-N$ times, we arrive at
the following inequality:
\begin{eqnarray}\label{th.N.pre.2}
\Pr\{W_n>x\} &\le& \Pr\{W_N>x-x_n-\ldots-x_{N+1}\}
+c_1\sum_{k=N+1}^n m^{k-1}\overline F(m^k(x-x_n-\ldots-x_{k+1}))\nonumber\\
&& +c_3\sum_{k=N+1}^n \overline F(x-x_n-\ldots-x_k)
m^{k-1}x\overline F(m^kx_k),
\end{eqnarray}
provided $x\le m^{(k-1)\delta}x_k^{1+\delta}$
for all $k=n$, \ldots, $N+1$.

Now take decreasing sequence $x_k=x^\gamma/k^2$.
Since $\gamma>1/2$ and $\delta\in(1/\gamma-1,1)$,
$x^{\gamma(1+\delta)}>x$.
Then \eqref{th.N.pre.2} holds for every $n\ge N+1$ and we have
\begin{eqnarray*}
\Pr\{W_n>x\} &\le& \Pr\{W_N>x-(1/(N+1)^2+\ldots+1/n^2)x^\gamma\}\\
&&+c_1\sum_{k=N+1}^n m^{k-1}
\overline F(m^k(x-(1/n^2+\ldots+1/(k+1)^2)x^\gamma))\\
&& +c_3\sum_{k=N+1}^n
\overline F(x-(1/n^2+\ldots+1/k^2)x^\gamma)
m^{k-1}x\overline F(m^kx^\gamma/k^2).
\end{eqnarray*}
Choose $N$ so large that $\sum_{k=N+1}^\infty 1/k^2\le 1$. Then
\begin{eqnarray*}
\Pr\{W_n>x\} &\le& \Pr\{W_N>x-x^\gamma\}
+c_1\overline F(x-x^\gamma)\sum_{k=N+1}^n m^{k-1}
\frac{\overline F(m^k(x-x^\gamma))}{\overline F(x-x^\gamma)}\\
&&\hspace{40mm} +c_3\overline F(x-x^\gamma)
\sum_{k=N+1}^n m^{k-1}x\overline F(m^kx^\gamma/k^2).
\end{eqnarray*}
Owing to the condition \eqref{cond:mat},
\begin{eqnarray*}
\sum_{k=N+1}^n m^{k-1}
\frac{\overline F(m^k(x-x^\gamma))}{\overline F(x-x^\gamma)}
&\le& c_4\sum_{k=N+1}^\infty
\frac{m^{k-1}}{m^{k(1+\delta)}}\to 0 \quad\mbox{as }N\to\infty
\end{eqnarray*}
and
\begin{eqnarray*}
\sum_{k=N+1}^n m^{k-1}x\overline F(m^kx^\gamma/k^2)
&\le& c_4 x\overline F(x^\gamma)
\sum_{k=N+1}^n \frac{m^{k-1}}{(m^k/k^2)^{1+\delta}}\\
&\le& c_4\E\xi^2 x^{1-2\gamma}\sum_{k=N+1}^\infty
\frac{m^{k-1}}{(m^k/k^2)^{1+\delta}}
\to 0 \quad\mbox{as }N\to\infty.
\end{eqnarray*}
Taking into account that
$\overline F(x-x^\gamma)\le c_5\overline F(mx)$
and further increasing $N$ we derive the following bound:
\begin{eqnarray*}
\Pr\{W_n>x\} &\le& \Pr\{W_N>x-x^\gamma\}
+\varepsilon\overline F(mx)/2.
\end{eqnarray*}
Applying here Lemma \ref{l:lower}, we deduce
$\overline F(mx)\le(1+o(1))\Pr\{W_N>x\}$ as $x\to\infty$, so
\begin{eqnarray*}
\Pr\{W_n>x\} &\le& (1+\varepsilon)\Pr\{W_N>x-x^\gamma\}
\end{eqnarray*}
for all sufficiently large $x$, and the proof is complete.
\end{proof}

Note that the assertion of Lemma \ref{l:upper} holds not only
for intermediate regularly varying distributions but for
Weibull distributions as well;
more precisely, the following result holds.

\begin{lemma}\label{l:upper.W}
Let $\overline F(x)=e^{-R(x)}$ where $R(x)$ satisfies
the condition \eqref{cond.non.increasing} and $R(x)/x\to0$.
Let the condition \eqref{cond:mat} hold. Then, for every
$\varepsilon>0$, there exists an $N$ such that
$$
\Pr\{W_n>x\} \le (1+\varepsilon)\Pr\{W_N>(1-\varepsilon)x\}
$$
for all $n>N$ and for all sufficiently large $x$.
\end{lemma}

\begin{proof}[Proof\/ {\rm is similar to that of Lemma \ref{l:upper}}]
Start again with the inequality \eqref{est.up}.
As follows from Proposition \ref{thm:upper.bound.for.sum.W}
for sums with zero mean, $\eta_i=\xi_i-m$, that,
for some $c<\infty$,
\begin{eqnarray*}
\Pr\Bigl\{\sum_{i=1}^k \xi_i>my\Bigr\} &=&
\Pr\Bigl\{\sum_{i=1}^k (\xi_i-m)>m(y-k)\Bigr\}\\
&\le& ck\overline F((m-\varepsilon/2)(y-k))
\quad\mbox{ for all }k\le z,
\end{eqnarray*}
provided $z\le\frac{(y-z)^2}{cR(y-z)}$. Therefore,
\begin{eqnarray*}
\Pr\Bigl\{\sum_{i=1}^{Z_{n-1}}\xi_i>my;Z_{n-1}\le z\Bigr\}
&\le& c\E Z_{n-1}\overline F((m-\varepsilon/2)(y-z))\\
&=& cm^{n-1}\overline F((m-\varepsilon/2)(y-z)).
\end{eqnarray*}
Substituting this into \eqref{est.up} with $y=m^{n-1}x$
and $z=m^{n-1}(x-x_n)$ we obtain
\begin{eqnarray*}
\Pr\{W_n>x\} &\le& \Pr\{W_{n-1}>x-x_n\}
+cm^{n-1}\overline F((m-\varepsilon/2)^n x_n),
\end{eqnarray*}
provided $x-x_n\le \frac{m^{n-1}x_n^2}{cR(m^{n-1}x_n)}$.
Iterating this upper bound $n-N$ times, we arrive at
the following inequality:
\begin{eqnarray}\label{th.N.pre*}
\Pr\{W_n>x\} &\le& \Pr\{W_N>x-x_n-\ldots-x_{N+1}\}
+c\sum_{k=N+1}^n m^{k-1}\overline F((m-\varepsilon/2)^k x_k),
\end{eqnarray}
provided $x-x_k\le \frac{m^{k-1}x_k^2}{cR(m^{k-1}x_k)}$
for all $k=n$, \ldots, $N+1$.
Take decreasing sequence $x_k=x/k^2$. Choose $N$ so large that
$\frac{m^{k-1}x}{k^2}\ge R(m^{k-1}x/k^2)$ for all $k\ge N+1$;
it is possible because $R(z)/z\to 0$ as $z\to\infty$.
Then \eqref{th.N.pre*} holds for every $n\ge N+1$ and we have
\begin{eqnarray*}
\Pr\{W_n>x\} &\le& \Pr\{W_N>(1-1/(N+1)^2-\ldots-1/n^2)x\}
+c\sum_{k=N+1}^n m^{k-1}\overline F((m-\varepsilon/2)^kx/k^2).
\end{eqnarray*}
Choose $\varepsilon>0$ so small to satisfy
$m<(m-\varepsilon/2)^{1+\delta}$ where $\delta>0$
is taken from the condition \eqref{cond:mat}.
Then the rest of the proof is the same as the proof of
Lemma \ref{l:upper}.
\end{proof}

\section{Finite time horizon asymptotics}
\label{sec:time}

As follows from \cite[Section 6]{DFK} for intermediate
regularly varying distribution $F$, for every fixed $n$,
\begin{eqnarray}\label{thm:DFK}
\Pr\{W_n>x\} &\sim& \sum_{i=0}^{n-1} m^i \overline F(m^{i+1}x)
\quad\mbox{ as }x\to\infty.
\end{eqnarray}
For the case where the second moment of $\xi$ is finite,
we extend this result for a wider class of distributions as follows.

\begin{lemma}\label{l:upper.sq}
Let $\E\xi^2<\infty$ and the distribution $F$ be dominated varying.
If $F$ is $x^\gamma$-insensitive for some $\gamma>1/2$,
then the equivalence \eqref{thm:DFK} holds for every fixed $n$.
\end{lemma}

\begin{proof}
First, Lemma \ref{l:lower.2} guarantees the right lower
bound. The upper bound will be proved by induction.
It is true for $n=1$. Assume, for some $n$,
\begin{eqnarray}\label{thm:time.upper}
\Pr\{W_n>x\} &\le& (1+o(1))\sum_{i=0}^{n-1} m^i\overline F(m^{i+1}x)
\quad \mbox{ as }x\to\infty.
\end{eqnarray}
Prove that then \eqref{thm:time.upper}
holds for $n+1$. Start with the inequality
\begin{eqnarray*}
\Pr\{W_{n+1}>x\} &=& \Pr\Bigl\{\sum_{i=1}^{Z_n}\xi_i>m^{n+1}x\Bigr\}\\
&\le& \Pr\{Z_n>m^n(x-x^\gamma)\}
+\Pr\Bigl\{\sum_{i=1}^{Z_n}\xi_i>m^{n+1}x;m^nx/2<Z_n\le m^n(x-x^\gamma)\Bigr\}\\
&&\hspace{40mm} +\Pr\Bigl\{\sum_{i=1}^{Z_n}\xi_i>m^{n+1}x;m^nx^\gamma<Z_n\le m^nx/2\Bigr\}\\
&&\hspace{60mm} +\Pr\Bigl\{\sum_{i=1}^{Z_n}\xi_i>m^{n+1}x;Z_n\le m^nx^\gamma\Bigr\}\\
&=:& P_1+P_2+P_3+P_4,
\end{eqnarray*}
where the $\xi$'s are independent of $Z_n$. Due to the induction
hypothesis and since $F$ is $x^\gamma$-insensitive,
\begin{eqnarray*}
P_1 &=& \Pr\{W_n>x-x^\gamma\}\\
&\le& (1+o(1))\sum_{i=0}^{n-1} m^i \overline F(m^{i+1}(x-x^\gamma))\\
&\sim& \sum_{i=0}^{n-1} m^i \overline F(m^{i+1}x)
\quad\mbox{ as }x\to\infty.
\end{eqnarray*}

Take $\delta\in(1/\gamma-1,1)$.
All the values of $k$ not greater than $m^nx$
are negligible compared to $y^{1+\delta}$
where $y=m^{n+1}x^\gamma$, $\gamma>1/2$.
Therefore, by Proposition \ref{thm:upper.bound.for.sum.R} there
exists $c<\infty$ such that
\begin{eqnarray*}
\Pr\Bigl\{\sum_{i=1}^k(\xi_i-m)>m^{n+1}x-km\Bigr\}
&\le& ck\overline F(m^{n+1}x-km)
\end{eqnarray*}
for sufficiently large $x$ and for all $k\le m^n(x-x^\gamma)$.

Therefore, for sufficiently large $x$,
\begin{eqnarray*}
P_2 &=& \sum_{k=m^nx/2}^{m^n(x-x^\gamma)}\Pr\{Z_n=k\}
\Pr\Bigl\{\sum_{i=1}^k(\xi_i-m)>m^{n+1}x-km\Bigr\}\\
&\le& c\sum_{k=m^nx/2}^{m^n(x-x^\gamma)}
\Pr\{Z_n=k\}k\overline F(m^{n+1}x-km)\\
&\le& cm^nx \Pr\{Z_n\ge m^nx/2\} \overline F(m^{n+1}x^\gamma).
\end{eqnarray*}
Since $\E\xi^2<\infty$ and $\gamma>1/2$,
$x \overline F(m^{n+1}x^\gamma)\to 0$ as $x\to\infty$.
Hence, as $x\to\infty$,
\begin{eqnarray*}
P_2 &=& o(\Pr\{W_n\ge x/2\})\\
&=& o(\overline F(mx/2))=o(\overline F(mx)),
\end{eqnarray*}
owing the induction hypothesis \eqref{thm:time.upper}
and dominated variation of $F$.

Further, for sufficiently large $x$,
\begin{eqnarray*}
P_3 &\le& c\sum_{k=m^nx^\gamma}^{m^nx/2}
\Pr\{Z_n=k\}k\overline F(m^{n+1}x-km)\\
&\le& c\E\{Z_n;Z_n>m^nx^\gamma\} \overline F(m^{n+1}x/2)\\
&=& o(\overline F(mx))\quad\mbox{ as }x\to\infty,
\end{eqnarray*}
again because of dominated variation of $F$.

Finally,
\begin{eqnarray*}
P_4 &=& \sum_{k=1}^{m^nx^\gamma}\Pr\{Z_n=k\}
\Pr\Bigl\{\sum_{i=1}^k(\xi_i-2m)>m^{n+1}x-2km\Bigr\}.
\end{eqnarray*}
The distribution $F$ is dominated varying and long-tailed
(constant-insensitive) which implies it belongs to the class
${\mathcal S}^*$, see, e.g. \cite[Theorem 3.29]{FKZ}.
Also, the expression $m^{n+1}x-2km$ tends to infinity as $x\to\infty$
uniformly in $k\le mx^\gamma$. This allows to apply here
Proposition \ref{S.star} for random variables $\eta_i:=\xi_i-2m$ with
negative mean; it ensures that, uniformly in $k\le mx^\gamma$,
\begin{eqnarray*}
\Pr\Bigl\{\sum_{i=1}^k(\xi_i-2m)>m^{n+1}x-2km\Bigr\}
&\le& (1+o(1))k\overline F(m^{n+1}x-2km)\\
&\sim& k\overline F(m^{n+1}x)\quad\mbox{ as }x\to\infty,
\end{eqnarray*}
because $F$ is $x^\gamma$-insensitive. Thus,
\begin{eqnarray*}
P_4 &\sim& \overline F(m^{n+1}x)
\sum_{k=1}^{m^nx^\gamma}\Pr\{Z_n=k\}k\\
&\sim& \overline F(m^{n+1}x)\E Z_n=m^n\overline F(m^{n+1}x).
\end{eqnarray*}

Combining bounds for $P_1$ through $P_4$ we deduce that
\begin{eqnarray*}
\Pr\{W_{n+1}>x\} &\le&
\Pr\{W_n>mx\}+m^n\overline F(m^{n+1}x)+o(\overline F(mx))
\quad\mbox{ as } x\to\infty,
\end{eqnarray*}
and the induction hypothesis \eqref{thm:time.upper}
completes the proof.
\end{proof}

If the distribution $F$ is rapidly varying then
\begin{eqnarray}\label{infty.1}
\sum_{i=0}^\infty m^i\overline F(m^{i+1}x)
&\sim& \overline F(mx)\quad\mbox{ as }x\to\infty.
\end{eqnarray}
Indeed, fix $\varepsilon>0$ and choose $x(\varepsilon)$ such that
$\overline F(mx)\le\varepsilon\overline F(x)$ for every $x>x(\varepsilon)$.
Then, for $x>x(\varepsilon)$,
$$
\sum_{i=1}^\infty m^i\overline F(m^{i+1}x)
\le \sum_{i=1}^\infty (m\varepsilon)^i\overline F(mx)
=\frac{m\varepsilon}{1-m\varepsilon}\overline F(mx).
$$
The constant multiplier on the right side may be made as small
as we please by appropriate choice of $\varepsilon$,
so the equivalence \eqref{infty.1} follows.

\begin{lemma}\label{l:upper.2}
Let $\overline F(x)=e^{-R(x)}$ where $R(x)$ is regularly
varying with index $\beta\in(0,1/2)$.
In the case $\beta\in[\frac{3-\sqrt 5}{2},1/2)$
assume also that the condition \eqref{cond:R.der} holds.
Additionally assume that $F\in{\mathcal S}^*$.
Then, for every fixed $n$,
$$
\Pr\{W_n>x\}\sim\overline F(mx)\quad\mbox{ as }x\to\infty.
$$
\end{lemma}

\begin{proof}
Since $\beta<1/2$, the distribution $F$ is $\sqrt x$-insensitive
which by Lemma \ref{l:lower.2} implies the lower bound
$\Pr\{W_n>x\}\ge(1+o(1))\overline F(mx)$ as $x\to\infty$.

To prove the upper bound, apply induction arguments.
For $n=1$, we have the equality $\Pr\{W_1>x\}=\overline F(mx)$.
Assume now $\Pr\{W_n>x\} \sim \overline F(mx)$
for some $n\ge 1$. Prove that then it holds for $n+1$.

If $\beta<\frac{3-\sqrt 5}{2}$ then the interval
$(\frac{1}{2-\beta},1-\beta)$ is not empty;
in this case we take $\gamma_1=\gamma_2\in(\frac{1}{2-\beta},1-\beta)$.
If $\beta\in[\frac{3-\sqrt 5}{2},1/2)$ then
$\frac{1}{2-\beta}\ge 1-\beta$ and we take
$\gamma_1\in(1/2,1-\beta)$ and $\gamma_2>1/(2-\beta)$ so that
$\gamma_2\ge\gamma_1$. Since $\gamma_1<1-\beta$,
$F$ is $x^{\gamma_1}$-insensitive. Start with the inequality
\begin{eqnarray*}
\Pr\Bigl\{\sum_{i=1}^{Z_n}\xi_i>m^{n+1}x\Bigr\}
&\le& \Pr\{Z_n>m^n(x-x^{\gamma_1})\}
+\Pr\Bigl\{\sum_{i=1}^{Z_n}\xi_i>m^{n+1}x;Z_n\le m^n(x-x^{\gamma_1})\Bigr\}\\
&=:& P_1+P_2,
\end{eqnarray*}
where the $\xi$'s do not depend on $Z_n$. By the induction
hypothesis and since $F$ is $x^{\gamma_1}$-insensitive,
$$
P_1\sim\overline F(m(x-x^{\gamma_1}))
\sim\overline F(mx)\quad\mbox{ as }x\to\infty.
$$
It remains to prove that $P_2=o(\overline F(mx))$ as $x\to\infty$.
Start with the following decomposition:
\begin{eqnarray*}
P_2 &=& \biggl(\sum_{k=1}^{m^n(x-x^{\gamma_2})-1}
+\sum_{k=m^n(x-x^{\gamma_2})}^{m^n(x-x^{\gamma_1})}\biggr)
\Pr\{Z_n=k\}
\Pr\Bigl\{\sum_{i=1}^k(\xi_i-m)>m^{n+1}x-km\Bigr\}\\
&=:& P_{21}+P_{22}.
\end{eqnarray*}
In the first sum $P_{21}$ we have
 $m^{n+1}x-mk\ge m^{n+1}x^{\gamma_2}\gg x^{1/(2-\beta)}$
 due to the choice $\gamma_2>1/(2-\beta)$.
 The function $R(x)/x^2$
is regularly varying with index $\beta-2$. Hence
$kR(m^{n+1}x-km)/(m^{n+1}x-km)^2\to 0$ as $x\to\infty$
uniformly in $k\le m^n(x-x^{\gamma_2})$.
This observation together with regular variation of $R(x)$
allows us to apply Proposition \ref{thm:upper.bound.for.sum.W}
with $y=(1-\varepsilon)x$ which ensures that
\begin{eqnarray*}
\Pr\Bigl\{\sum_{i=1}^k(\xi_i-m)>m^{n+1}x-mk\Bigr\}
&\le& k\overline F((m^{n+1}x-mk)(1-\varepsilon))
\end{eqnarray*}
for sufficiently large $x$ and for all $k\le m^n(x-x^{\gamma_2})$.
Thus, for sufficiently large $x$,
\begin{eqnarray*}
P_{21} &\le& \sum_{k=1}^{m^n(x-x^{\gamma_2})}\Pr\{Z_n=k\}k
\overline F((m^nx-k)m(1-\varepsilon)).
\end{eqnarray*}
Take $\varepsilon>0$ so small that $m(1-\varepsilon)>1$.
Then by rapid variation of $F$, as $x\to\infty$,
$$
\overline F((m^nx-k)m(1-\varepsilon))
= o(\overline F(m^nx-k))
\quad\mbox{uniformly in }k\le m^n(x-x^{\gamma_2}).
$$
In addition, owing the induction hypothesis,
\begin{eqnarray*}
\Pr\{Z_n=k\} &\le& \Pr\{W_n\ge k/m^n\} \le c\overline F(k/m^{n-1}),
\end{eqnarray*}
for some $c<\infty$. Thus, as $x\to\infty$,
\begin{eqnarray*}
P_{21} &\le& o(1)\int_0^{m^n(x-x^{\gamma_2})}
y\overline F(y/m^{n-1})\overline F(m^nx-y)dy\\
&=& o(1)\int_0^{m(x-x^{\gamma_2})}
y\overline F(y)\overline F(m^{n-1}(mx-y))dy.
\end{eqnarray*}
Since $m^{n-1}\ge m>1$ and $\beta>0$,
$$
y\overline F(m^{n-1}(mx-y))=o(\overline F(mx-y))
$$
as $x\to\infty$ uniformly in $y\le m(x-x^{\gamma_2})$. Therefore,
\begin{eqnarray*}
P_{21} &\le& o(1)\int_0^{mx} \overline F(y)\overline F(mx-y)dy.
\end{eqnarray*}
The inclusion $F\in{\mathcal S}^*$ means that
\begin{eqnarray*}
\int_0^{mx} \overline F(y)\overline F(mx-y)dy
\sim 2\overline F(mx)\int_0^\infty \overline F(y)dy
\quad\mbox{ as }x\to\infty,
\end{eqnarray*}
which finally implies $P_{21}=o(\overline F(mx))$.
In the case $\beta<\frac{3-\sqrt 5}{2}$ this completes the proof
because then $\gamma_1=\gamma_2$ and $P_{22}=0$.

If $\beta<1/2$ then it remains to prove that
$P_{22}=o(\overline F(mx))$ too. We have
\begin{eqnarray*}
P_{22} &=& \sum_{k=m^n x^{\gamma_1}}^{m^n x^{\gamma_2}}
\Pr\{Z_n=m^nx-k\}
\Pr\Bigl\{\sum_{i=1}^{m^n x-k}(\xi_i-m)>mk\Bigr\}.
\end{eqnarray*}
By the induction hypothesis
$$
\Pr\{Z_n=m^nx-k)\}\le\Pr\{W_n\ge x-k/m^n\}
\sim\overline F(mx-k/m^{n-1}),
$$
so that
\begin{eqnarray*}
P_{22} &\le& c_1\sum_{k=m^nx^{\gamma_1}}^{m^nx^{\gamma_2}}
\overline F(mx-k/m^{n-1}) (m^nx-k+1) \overline F(y_k)
\end{eqnarray*}
for any $y_k$ satisfying the inequalities
$y_k\le mk/2$ and $m^nx-k\le (mk)^2/cR(y_k)$,
where $c=c(1/2)$ is defined in Proposition \ref{thm:upper.bound.for.sum.W}.
Choose $\gamma\in(2\beta,1)$ such that
\begin{equation}\label{choice.gamma}
\frac{1}{\gamma_1}-1<\gamma<\frac{1}{\gamma_2}-1+\beta,
\end{equation}
it is possible if we choose $\gamma_2>1/(2-\beta)$ sufficiently
close to $1/(2-\beta)$. Then take $y_k$ which solves
$R(y_k) = m^{2-n}k^{1+\gamma}/cx=c_2 k^{1+\gamma}/x$.
With this choice, $y_k\le mk/2$ for $k\le m^nx^{\gamma_2}$
and sufficiently large $x$,
by the right inequality in \eqref{choice.gamma},
and $m^nx-k\le (mk)^2/cR(y_k)$.

Further, since $\overline F(y_k)=e^{-R(y_k)}$,
\begin{eqnarray*}
P_{22} &\le& c_3 x \sum_{k=m^nx^{\gamma_1}}^{m^nx^{\gamma_2}}
\overline F(mx-k/m^{n-1}) \overline F(y_k)\\
&\le& c_3 x \overline F(mx)\sum_{k=m^nx^{\gamma_1}}^{m^nx^{\gamma_2}}
e^{R(mx)-R(mx-k/m^{n-1})-R(y_k)}.
\end{eqnarray*}
By the condition \eqref{cond:R.der} on the increments
of $R$ and by regular variation of $R$ we have
\begin{eqnarray*}
\frac{R(x)-R(y)}{x-y} &\le& c_4\frac{R(x)}{x},\quad x\ge y\ge 1,
\end{eqnarray*}
which implies
\begin{eqnarray*}
R(mx)-R(mx-k/m^{n-1})-R(y_k) &\le& c_5 k R(mx)/x-c_2k^{1+\gamma}/x\\
&=& (c_5R(mx)-c_2k^\gamma)k/x.
\end{eqnarray*}
 Since $R(mx)$ is regularly varying with index $\beta<1/2$
and $k\ge m^n x^{\gamma_1}$, the choice $\gamma_1\in(1/2,1-\beta)$
and $\gamma\in(2\beta,1)$ ensures $R(mx)=o(k)$. Hence,
\begin{eqnarray*}
R(mx)-R(mx-k/m^{n-1})-R(y_k) &\le& -c_6k^{1+\gamma}/x,
\end{eqnarray*}
which yields
\begin{eqnarray*}
P_{22} &\le&
c_4 x \overline F(mx)\sum_{k=m^nx^{\gamma_1}}^\infty
e^{-c_6k^{1+\gamma}/x} = o(\overline F(mx))\quad\mbox{as }x\to\infty,
\end{eqnarray*}
due to $\gamma_1(1+\gamma)>1$, by the left inequality in \eqref{choice.gamma}.
Combining altogether we deduce that
$P_2=o(\overline F(mx))$ and consequently
$\Pr\{W_{n+1}>x\} \sim P_1\sim \overline F(mx)$ as $x\to\infty$
and the proof is complete.
\end{proof}

\section{Proofs of Theorems \ref{thm:irv}, \ref{thm:sq.root}
and \ref{thm:asy.uni.W}}
\label{sec:proofs}

\begin{proof}[Proof of Theorem {\rm\ref{thm:irv}}]
The bounds \eqref{domin.b} follow from Lemma \ref{l:lower}
and Corollary \ref{cor:upper.D}.
All other assertions follow from the equivalence \eqref{thm:DFK}
and from Lemmas \ref{l:lower} and \ref{l:upper}.
\end{proof}

\begin{proof}[Proof of Theorem {\rm\ref{thm:sq.root}
follows from Lemmas \ref{l:upper.sq}, \ref{l:lower.2}
and \ref{l:upper.sq.N}}]
\end{proof}

\begin{proof}[Proof of Theorem {\rm\ref{thm:asy.uni.W}}]
The lower bound for the general case $\beta<1$
follows from Lemma \ref{l:lower}.
The upper bound follows from Lemma \ref{l:upper.W}
which reduces the problem to the finite time horizon $N$
and further induction arguments like
\begin{eqnarray*}
\Pr\{W_N>x\} &=& \Pr\Bigl\{\sum_{i=1}^\xi W_{N-1}^{(i)}>mx\Bigr\}\\
&\le& \Pr\{\xi>mx(1-\varepsilon)\}
+\Pr\Bigl\{\sum_{i=1}^\xi W_{N-1}^{(i)}>mx;\xi\le mx(1-\varepsilon)\Bigr\},
\end{eqnarray*}
where $W_{N-1}^{(1)}$, $W_{N-1}^{(2)}$, \ldots\ are independent
copies of $W_{N-1}$. Assuming that $W_{N-1}$ has a tail
not heavier than $c\overline F((1-\varepsilon)x)$  me may
estimate here the second probability as follows:
\begin{eqnarray*}
\Pr\Bigl\{\sum_{i=1}^\xi W_{N-1}^{(i)}>mx;\xi\le mx(1-\varepsilon)\Bigr\}
&=& \sum_{k=1}^{mx(1-\varepsilon)}\Pr\{\xi=k\}
\Pr\Bigl\{\sum_{i=1}^k (W_{N-1}^{(i)}-k)>mx-k\Bigr\}.
\end{eqnarray*}
By Proposition \ref{thm:upper.bound.for.sum.W},
\begin{eqnarray*}
\Pr\Bigl\{\sum_{i=1}^k (W_{N-1}^{(i)}-k)>mx-k\Bigr\}
&\le& (k+1)\overline F((1-\varepsilon)(mx-k))
\end{eqnarray*}
as $x\to\infty$ uniformly in $k\le mx(1-\varepsilon)$;
note that the condition $k\le mx(1-\varepsilon)$ implies
$mx-k\ge mx\varepsilon$ and hence covers
both conditions of Proposition \ref{thm:upper.bound.for.sum.W}.
Thus,
\begin{eqnarray*}
\sum_{k=1}^{mx(1-\varepsilon)}\Pr\{\xi=k\}
\Pr\Bigl\{\sum_{i=1}^k (W_{N-1}^{(i)}-k)>mx-k\Bigr\}
&\le& 2\sum_{k=1}^{mx(1-\varepsilon)}\Pr\{\xi=k\}k
\overline F((1-\varepsilon)(mx-k))\\
&=& o(\overline F(m(1-2\varepsilon)x))
\end{eqnarray*}
as $x\to\infty$, by standard properties of Weibull type distributions.
This completes the proof of upper bound for the case $\beta<1$.

In the case $\beta<1/2$ the distribution $F$ is $\sqrt x$-insensitive
which by Lemma \ref{l:lower.2} implies the lower bound
$\Pr\{W_n>x\}\ge(1+o(1))\overline F(mx)$ as $x\to\infty$.

Now prove the upper bound for the case $\beta<1/2$.
Fix $\varepsilon>0$. Owing Lemma \ref{l:upper.W} we find $N$
so that, for all $n>N$ and for all sufficiently large $x$,
$$
\Pr\{W_n>x\} \le (1+\varepsilon)\Pr\{W_N>(1-\varepsilon)x\}.
$$
As in the proof of Lemma \ref{l:upper.2},
take $\gamma\in(1/(2-\beta),1-\beta)$ so $F$ is $x^\gamma$-insensitive.
Make use of the decomposition, for $n>N+1$,
\begin{eqnarray*}
\Pr\{W_n>x\} &\le& \Pr\{\xi>m(x-x^\gamma)\}
+\Pr\Bigl\{\sum_{i=1}^{\xi} W_{n-1}^{(i)}>mx;\xi\le m(x-x^\gamma)\Bigr\}\\
&=:& P_1+P_2.
\end{eqnarray*}
Since $F$ is $x^\gamma$-insensitive,
$$
P_1=\Pr\{\xi>m(x-x^\gamma)\}\sim\overline F(mx)\quad\mbox{ as }x\to\infty.
$$
Further, make use of Lemma \ref{l:upper.W} which is applicable
because $n-1>N$: ultimately in $y$,
\begin{eqnarray*}
\Pr\{W_{n-1}>y\} &\le& (1+\varepsilon)\Pr\{W_N>(1-\varepsilon)y\}\\
&\le& (1+2\varepsilon)\overline F((1-\varepsilon)my),
\end{eqnarray*}
by virtue of Lemma \ref{l:upper.2}.
Choose $\varepsilon>0$ so small
that $m_*:=(1-\varepsilon)m>1$, it is possible due to $m>1$.
The family $\{W_{n-1}-1,n>N+1\}$ satisfies the conditions
of Corollary \ref{cor:upper.bound.for.sum.W} which further
allows to prove that $P_2=o(\overline F(mx))$ as $x\to\infty$
uniformly in $n>N+1$
in the same way as in the proof of Lemma \ref{l:upper.2}.
\end{proof}

\section{The case of regularly varying tail with index $-1$;
proof of Theorem \ref{T10}}
\label{sec:first.moment}

As proven in Lemma \ref{l:lower}, for every $\varepsilon>0$,
$$
\Pr\{W_n>x\} \ge (1+o(1))\sum_{k=0}^{n-1} m^k
\overline F(m^{k+1}(1+\varepsilon)x)
$$
as $x\to\infty$ uniformly in $n\ge 1$.
Since $F$ is regularly varying, we deduce from here that
$$
\Pr\{W_n>x\} \ge (1+o(1))\sum_{k=0}^{n-1} m^k
\overline F(m^{k+1}x)
$$
as $x\to\infty$ uniformly in $n\ge 1$.
Then it remains to prove the following upper bound:
for every fixed $\varepsilon>0$,
\begin{eqnarray}\label{T10.upper}
\Pr\{W_n>x\} &\le& (1+o(1))\sum_{k=0}^{n-1} m^k
\overline F(m^{k+1}x(1-\varepsilon)).
\end{eqnarray}
The method for proving upper bounds based on
Lemma \ref{l:upper} doesn't work here because it essentially
requires the condition \eqref{cond:mat}. By this reason we
proceed in a different way. Define events
$$
A_k(x):=\{\xi_i^{(k)}>m^{k+1}x(1-\varepsilon)
\mbox{ for some }i\le Z_k\}.
$$
Clearly,
$$
\Pr\{A_k(x)|Z_k=j\}\leq j \overline F(m^{k+1}x(1-\varepsilon)),\qquad j\ge 1.
$$
Therefore, $\Pr\{A_k(x)\}\leq m^k \overline F(m^{k+1}x(1-\varepsilon))$ and
$$
\Pr\Bigl\{\bigcup_{k=0}^{n-1}A_k(x)\Bigr\}
\le \sum_{k=0}^{n-1} \Pr\{A_k(x)\}
\le\sum_{k=0}^{n-1} m^k \overline F(m^{k+1}x(1-\varepsilon)).
$$
Owing to this and the upper bound
\begin{eqnarray*}
\Pr\{W_n>x\} &\le& \Pr\Bigl\{W_n>x,\bigcap_{k=0}^{n-1}
\overline{A_k(x)}\Bigr\}
+\Pr\Bigl\{\bigcup_{k=0}^{n-1}A_k(x)\Bigr\},
\end{eqnarray*}
we conclude that \eqref{T10.upper} will be implied by
the following relation: for every fixed $\varepsilon>0$,
\begin{eqnarray}\label{T10.ttancated}
\Pr\Bigl\{W_n>x,\bigcap_{k=0}^{n-1}\overline{A_k(x)}\Bigr\}
&=& o\Bigl(\sum_{k=0}^{n-1} m^k \overline F(m^{k+1}x)\Bigr)
\end{eqnarray}
as $x\to\infty$ uniformly in $n\ge 1$.
By the Chebyshev inequality, for every $\lambda>0$
\begin{eqnarray*}
\Pr\Bigl\{W_n>x,\bigcap_{k=0}^{n-1}\overline{A_k(x)}\Bigr\}
&\le& \frac{\E\bigl\{e^{\lambda Z_n}-1;\bigcap_{k=0}^{n-1}\overline{A_k(x)}\bigr\}}
{e^{\lambda m^n x}-1}\\
&=& \frac{\E\bigl\{e^{\lambda Z_n};\bigcap_{k=0}^{n-1}\overline{A_k(x)}\bigr\}-1}
{e^{\lambda m^n x}-1}
+\frac{\Pr\bigl\{\bigcup_{k=0}^{n-1} A_k(x)\bigr\}}
{e^{\lambda m^n x}-1},
\end{eqnarray*}
so that the relation \eqref{T10.ttancated} will follow
if we find $\lambda=\lambda_n(x)$ such that
\begin{eqnarray}\label{T10.ttancated.2}
\lambda m^n x &\to& \infty
\end{eqnarray}
and
\begin{eqnarray}\label{T10.ttancated.1}
\frac{\E\bigl\{e^{\lambda Z_n};\bigcap_{k=0}^{n-1}\overline{A_k(x)}\bigr\}-1}
{e^{\lambda m^n x}-1}
&=& o\Bigl(\sum_{k=0}^{n-1} m^k \overline F(m^{k+1}x)\Bigr).
\end{eqnarray}
In order to find $\lambda=\lambda_n(x)$
satisfying \eqref{T10.ttancated.2} and \eqref{T10.ttancated.1}
we proceed with a suitable exponential
bounds for bounded random variables. Take $\lambda_{nn}>0$ and
consider the following exponential moment
\begin{eqnarray*}
\E\Bigl\{e^{\lambda_{nn} Z_n};\bigcap_{k=0}^{n-1}
\overline{A_k(x)}\Bigr\}
&=& \sum_{i=1}^\infty \E\Bigl\{e^{\lambda_{nn} Z_n};
\bigcap_{k=0}^{n-1} \overline{A_k(x)},Z_{n-1}=i\Bigr\}\\
&=& \sum_{i=1}^\infty \E\Bigl\{
e^{\lambda_{nn}(\xi_1^{(n-1)}+\ldots+\xi_i^{(n-1)})};
\overline{A_{n-1}(x)},
\bigcap_{k=0}^{n-2} \overline{A_k(x)},Z_{n-1}=i\Bigr\}.
\end{eqnarray*}
Note that the events $\bigcap_{k=0}^{n-2} \overline{A_k(x)}$ and
$Z_{n-1}=i$ do not depend on the $\xi^{(n-1)}$'s. Therefore,
\begin{eqnarray*}
\E\Bigl\{e^{\lambda_{nn} Z_n};\bigcap_{k=0}^{n-1}
\overline{A_k(x)}\Bigr\}
&=& \sum_{i=1}^\infty \E\Bigl\{
e^{\lambda_{nn}(\xi_1^{(n-1)}+\ldots+\xi_i^{(n-1)})};
\overline{A_{n-1}(x)}\Bigr\}
\Pr\Bigl\{\bigcap_{k=0}^{n-2} \overline{A_k(x)},Z_{n-1}=i\Bigr\}\\
&=& \sum_{i=1}^\infty \Bigl(\E\{e^{\lambda_{nn}\xi};\xi\le m^n x(1-\varepsilon)\}
\Bigr)^i
\Pr\Bigl\{\bigcap_{k=0}^{n-2} \overline{A_k(x)},Z_{n-1}=i\Bigr\}.
\end{eqnarray*}
If we put
\begin{eqnarray*}
\lambda_{n,n-1} &:=& \log \E\{e^{\lambda_{nn}\xi};
\xi\le m^n x(1-\varepsilon)\},
\end{eqnarray*}
then we receive a recursive equality
\begin{eqnarray*}
\E\Bigl\{e^{\lambda_{nn} Z_n};\bigcap_{k=0}^{n-1}
\overline{A_k(x)}\Bigr\}
&=& \E\Bigl\{e^{\lambda_{n,n-1}Z_{n-1}};
\bigcap_{k=0}^{n-2} \overline{A_k(x)}\Bigr\}.
\end{eqnarray*}
We iterate this recursion $n$ times. Let us
estimate $\lambda_{n,n-1}$ via $\lambda_{nn}$.

For every $z>0$ and $y\le z$, $e^y \le 1+y+y^2e^z/2$.
Therefore,
\begin{eqnarray*}
\E\{e^{\lambda_{nn}\xi};\xi\le m^n x(1-\varepsilon)\} &\le&
1+\lambda_{nn}m+\lambda_{nn}^2 \E\{\xi^2;\xi\le m^n x\}
e^{\lambda_{nn}m^nx(1-\varepsilon)}/2.
\end{eqnarray*}
Since $F$ is regularly varying with index $-1$,
for sufficiently large $x$,
\begin{eqnarray*}
\E\{\xi^2;\xi\le m^nx\} &\le& \frac{3}{2}(m^n x)^2\overline F(m^n x).
\end{eqnarray*}
Hence,
\begin{eqnarray}\label{lambda.n-1.1}
\E\{e^{\lambda_{nn}\xi};\xi\le m^n x(1-\varepsilon)\}
&\le& 1+\lambda_{nn}\Bigl(m+\frac{3}{4}
(\lambda_{nn} m^n x)m^n x\overline F(m^n x)
e^{\lambda_{nn}m^nx(1-\varepsilon)}\Bigr).
\end{eqnarray}
Denote
$$
p_n(x) := \sum_{k=0}^{n-1}m^k\overline F(m^{k+1}x)
$$
and make a special choice of initial $\lambda_n$:
$$
\lambda_{nn}=\lambda_n(x):=(1+\varepsilon)\frac{\log\frac{1}{p_n(x)x}
-2\log\log\frac{1}{p_n(x)x}}
{x\prod_{k=0}^{n-1}\Bigl(m+\frac{m^{k+1}x\overline F(m^{k+1}x)}
{(p_n(x)x)^{1-\varepsilon^2}}\Bigr)}.
$$
For the product, we have the following inequalities:
\begin{eqnarray*}
m^n \le \prod_{k=0}^{n-1}\Bigl(m+\frac{m^{k+1}x\overline F(m^{k+1}x)}
{(p_n(x)x)^{1-\varepsilon^2}}\Bigr)
&=& m^n \prod_{k=0}^{n-1}\Bigl(1+\frac{m^k \overline F(m^{k+1}x)}
{p_n(x)}(p_n(x)x)^{\varepsilon^2}\Bigr)\\
&\le& m^n e^{(p_n(x)x)^{\varepsilon^2}
\sum_{k=0}^{n-1}\frac{m^k \overline F(m^{k+1}x)}{p_n(x)}}\\
&=& m^n e^{(p_n(x)x)^{\varepsilon^2}}.
\end{eqnarray*}
Note that then this product is asymptotically equivalent to
$m^n$ because $p_n(x)x\to 0$. Note also that then
\begin{eqnarray}\label{devs}
\frac{1}{e^{\lambda_{nn} m^nx}-1}
&\le& \frac{1}{\bigl(\frac{1}{p_n(x)x}\bigr)^{1+\varepsilon+o(1)}
\log^{-2(1+\varepsilon+o(1))}\frac{1}{p_n(x)x}-1}\nonumber\\
&\sim& \bigl(p_n(x)x\bigr)^{1+\varepsilon+o(1)}
\log^{2(1+\varepsilon+o(1))}\frac{1}{p_n(x)x}\nonumber\\
&\le& c_1(p_n(x)x)^{1+\varepsilon/2}
\end{eqnarray}
ultimately in $x$ uniformly in $n$. In particular, it goes
to zero and the relation \eqref{T10.ttancated.2} follows.

Now estimate all $\lambda_{nk}$, $k\le n-1$.
With the choice of $\lambda_{nn}$ made,
it follows from \eqref{lambda.n-1.1} that
\begin{eqnarray*}
\E\{e^{\lambda_{nn}\xi};\xi\le m^n x(1-\varepsilon)\}
&\le& 1+\lambda_{nn}\Bigl(m+
\frac{3(1+\varepsilon)}{4}\log\frac{1}{p_n(x)x}\\
&&\hspace{20mm} \times m^n x\overline F(m^n x)
e^{(1-\varepsilon^2)\bigl(\log\frac{1}{p_n(x)x}
-2\log\log\frac{1}{p_n(x)x}\bigr)}\Bigr)\\
&\le& 1+\lambda_{nn}\Bigl(m+ m^nx\overline F(m^n x)
e^{(1-\varepsilon^2)\log\frac{1}{p_n(x)x}}\Bigr),
\end{eqnarray*}
provided $1+\varepsilon<4/3$ and $2(1-\varepsilon^2)>1$. Thus,
\begin{eqnarray*}
\E\{e^{\lambda_{nn}\xi};\xi\le m^n x(1-\varepsilon)\}
&\le& 1+\lambda_{nn}\Bigl(m+ \frac{m^nx\overline F(m^n x)}
{(p_n(x)x)^{1-\varepsilon^2}}\Bigr)\\
&\le& \exp\Bigl\{\lambda_{nn}\Bigl(m+ \frac{m^nx\overline F(m^n x)}
{(p_n(x)x)^{1-\varepsilon^2}}\Bigr)\Bigr\},
\end{eqnarray*}
which yields
\begin{eqnarray*}
\lambda_{n,n-1} &\le& \lambda_{nn}\Bigl(m+ \frac{m^nx\overline F(m^n x)}
{(p_n(x)x)^{1-\varepsilon^2}}\Bigr)
=(1+\varepsilon)\frac{\log\frac{1}{p_n(x)x}
-2\log\log\frac{1}{p_n(x)x}}
{x\prod_{k=0}^{n-2}\Bigl(m+\frac{m^{k+1}x\overline F(m^{k+1}x)}
{(p_n(x)x)^{1-\varepsilon^2}}\Bigr)},
\end{eqnarray*}
Iterating this estimate $n$ times we finally deduce that
\begin{eqnarray*}
\E\Bigl\{e^{\lambda_{nn} Z_n};\bigcap_{k=0}^{n-1}
\overline{A_k(x)}\Bigr\}
&\le& \E e^{\lambda_{n0}}=e^{\lambda_{n0}}
= \exp\Bigr\{\frac{1+\varepsilon}{x}\log\frac{1}{p_n(x)x}\Bigr\}.
\end{eqnarray*}
From here and \eqref{devs},
\begin{eqnarray*}
\frac{\E\bigl\{e^{\lambda_{nn} Z_n};\bigcap_{k=0}^{n-1}\overline{A_k(x)}\bigr\}-1}
{e^{\lambda_{nn} m^n x}-1}
&\le& c_2x^{-1}(p_n(x)x)^{1+\varepsilon/2}\log\frac{1}{p_n(x)x}\\
&\le& c_3x^{-1}(p_n(x)x)^{1+\varepsilon/4}\\
&=& c_3p_n(x)(p_n(x)x)^{\varepsilon/4}=o(p_n(x))
\end{eqnarray*}
and \eqref{T10.ttancated.1} is also proven.
This completes the proof of Theorem \ref{T10}.

\section*{Acknowledgment}

We are greatly grateful to both anonymous referees who
pointed out some incorrectness in the first version of the paper
as well as for their very positive comments.

\end{document}